\documentclass[12pt]{article}  
\usepackage{amssymb}
\usepackage{amsmath,amsthm,enumerate}
\theoremstyle{plain}
\newtheorem{thm}{Theorem}[section]

\theoremstyle{definition}

\newtheorem{rmk}[thm]{Remark}

\title{Generalized Riemann hypotheses: sufficient and equivalent criteria\thanks{Research
    supported by Swiss National Science Foundation Grant no. 107887. Support has also been given over time by Scuola Normale Superiore of Pisa, University of Z\"urich and IBM Z\"urich Research Lab.}}
\author{Davide Schipani\\
Institute of Mathematics\\
University of Zurich\\
davide.schipani(at)math.uzh.ch}
\date{\today}
\begin{document}\maketitle

\begin{abstract}


This paper presents new sufficient and equivalent conditions for the generalized version of the Riemann Hypothesis. The paper derives also statements and remarks concerning zero-free regions, modified Hadamard-product formulas and the behaviour of $\left|\frac{L(\bar{\chi},s)}{L(\chi,1-s)}\right|$.

\end{abstract}\medskip

\section{Introduction}

The Riemann zeta function $\zeta(s)$ and the Riemann Hypothesis have been the object of a lot of generalizations and there is a growing literature in this regard comparable with that of the classical zeta function itself. We presented in \cite{sc10a} some sufficient and equivalent criteria for the classical case and we would like here to show that they are extendable in many situations. We are not covering all possible generalizations, which are fairly many, but we will concentrate essentially on the main one, with a few remarks on some other scenarios.

An overview of the paper is the following. In this section we give a short introduction to the main generalizations of $\zeta$ and the corresponding generalized conjectures. In Section $2$ we provide some remarks about zero-free regions and modified Hadamard-product formulas, while in Section $3$ we deal with analogues of some sufficient or equivalent criteria given in \cite{sc10a}. 

The most direct generalization, which is also what we will mainly deal with, concerns the Dirichlet $L$-functions with the corresponding Extended Riemann Hypothesis. 

We recall first, following \cite{na00}, that a Dirichlet character $\chi$ modulo $q$ is a complex-valued function on the integers with the following properties: $\chi(n)\neq 0$ if and only if $(n,q)=1$; $\chi(m)=\chi(n)$ if $m\equiv n$ (mod $q$); $\chi(mn)=\chi(m)\chi(n)$ for all integers $m$ and $n$. The principal character is the one such that $\chi(n)=1$ whenever $(n,q)=1$.

A character $\chi$ is even if $\chi(-1)=1$, otherwise it is odd.

If $d$ and $q$ are positive integers, with $d$ dividing $q$, and if $\chi$ is a Dirichlet character modulo $d$, then we can define an induced character modulo $q$ by $\chi_{*}(n)=\chi(n)$ if $(n,q)=1$ and $\chi_{*}(n)=0$ if $(n,q)\neq 1$.


A character $\chi$ modulo $q$ is called primitive if it is not induced from a character modulo $d$ for any proper divisor $d$ of $q$. In particular the unique Dirichlet character modulo $1$, which takes the value $1$ for all integers, induces the principal character modulo $q$ for every $q>1$. The conductor of a character $\chi$ modulo $q$ is the smallest positive divisor $f$ of $q$ such that $\chi$ is induced from a character modulo $f$.

Now, a Dirichlet L-function $L(\chi,s)$ is defined for $\Re(s)>1$ as $\sum_{n=1}^{\infty}\frac{\chi(n)}{n^s}$; this series is convergent also for $\Re(s)>0$, except in the case of a principal character (see e.g. \cite{wi46}).

If $\chi$ is a character modulo $q$ with conductor $f$ and corresponding primitive character $\chi_f$ modulo $f$, then we have (see e.g. \cite{co07}):
$$
L(\chi,s)=L(\chi_f,s)\prod_{p|q}\left(1-\frac{\chi_f(p)}{p^s}\right)
$$
In particular,  for a principal character $\chi_0$ we have
$$
L(\chi_0,s)=\zeta(s)\prod_{p|q}\left(1-\frac{1}{p^s}\right).
$$
If $\chi$ is not principal and $q$ is a prime power, then $L(\chi,s)=L(\chi_f,s)$.

Following \cite[Corollary 10.2.15 and Definition 10.2.16]{co07}, if $e$ is $0$ or $1$ such that $\chi(-1)=(-1)^e$, we call extraneous zeros the zeros of the product above in the expression for $L(\chi,s)$ for a nonprimitive character $\chi$, and trivial zeros the zeros $s=e-2k$ with $k$ positive integer, as well as $s=0$ in case $\chi$ is a nonprincipal even character. All other zeros are called nontrivial. 

The Extended Riemann Hypothesis conjectures that the Dirichlet $L$-functions have all their nontrivial zeros on the critical line $\Re(s)=\frac{1}{2}$ or, in other words, that $L(\chi,s)$, for a
primitive character $\chi$ modulo $q$, has no zeros
with real part different from $1/2$ in the critical strip $0<\Re(s)<1$, since we can exclude nontrivial zeros outside (and the strong version says that $L(\chi,1/2)$ is always
nonzero too; see also \cite[secton 10.5.7]{co07}).

What we know about zeros in the critical strip is at least that there is an infinite number of them on the critical line (\cite{bo95a}, \cite{mu08}, \cite{zu78}, \cite{zu78a}), which is the analogue of Hardy's theorem about the number of roots of $\zeta$ (see e.g. \cite[chapter 11]{ed01}).

\vspace{1cm}

We introduce now briefly the Dedekind zeta function, since it is linked to Dirichlet L-functions (see the following Remark). If $K$ is a number field and $\mathbb{Z}_K$ its ring of integers, then the Dedekind zeta function $\zeta_K(s)$ is defined as (see e.g. \cite[section 10.5.1]{co07}):
$$
\zeta_K(s)=\sum_{a\subset\mathbb{Z}_K} \frac{1}{N(a)^s},
$$
where $a$ runs through all integral ideals of $\mathbb{Z}_K$ and $N(a)$ denotes their absolute norm.

The corresponding conjecture is usually called the Generalized Riemann Hypothesis and says that the roots in the critical strip have real part equal to $1/2$.

\begin{rmk}
Notice that the Dedekind zeta function can be represented, whenever the field involved is an abelian extension of
$\mathbb{Q}$, as a product of Dirichlet L-functions (see
e.g. \cite[Theorem 10.5.25]{co07}): then proving the Extended Riemann Hypothesis would provide a proof of the Generalized Riemann Hypothesis for this
kind of fields too. And notice that proving the Extended Riemann Hypothesis only for the group of real
characters would prove on its own the Generalized Riemann Hypothesis for the field associated to
this group: see \cite[Theorem 4.3]{wa82}.
\end{rmk}
 
\section{Zero-free regions and product formulas}

We make here some remarks about product expansions of $\zeta$ and its generalizations. These are products over the roots of these functions and appear often to be of interest when considering zeros or zero-free regions.

We start by reviewing the classical case: the function $\xi(s)\doteq\frac{s(s-1)}{2}\pi^{-\frac{s}{2}}\Gamma(\frac{s}{2})\zeta(s)$ is an
  entire function of finite order and admits the following infinite
  Hadamard-product expansion which ranges through all the nontrivial roots $\rho_n$ of $\zeta$:
$$
\xi(s)=e^{A+Bs}\prod_{n=1}^{\infty}\left(1-\frac{s}{\rho_n}\right)e^{s/\rho_n},
$$
for some constants $A$ and $B$, the product being absolutely and
 uniformly convergent on compact subsets (see e.g. \cite{da00}, \cite{ka92}, \cite{ba97}).

There are simpler expressions: as shown in
\cite{ed01}
$$
\xi(s)=c_1\prod\left(1-\frac{s}{\rho_n}\right),
$$
or
$$
\xi(s)=c_2\prod\left(1-\frac{s-\frac{1}{2}}{\rho_n-\frac{1}{2}}\right),
$$
where $c_1,c_2$ are some constants and the product is understood to be taken over all roots $\rho_n$ of $\xi$ (which are the nontrivial roots of $\zeta$) in an order which pairs each root $\rho_n$ with the corresponding one $1-\rho_n$.
Considering $s=\frac{1}{2}+iv$ and substituting $\phi_n=i(\frac{1}{2}-\rho_n)$, we can also write
$$
\Xi(v)\doteq\xi(\frac{1}{2}+iv)=c_2\prod\left(1-\frac{v}{\phi_n}\right),
$$
What is more common in the literature and brings the
advantage of forgetting about the order is to pair $\rho_n$ with
$1-\rho_n$, which corresponds to $-\phi_n$, so that we can write the
following product representation over all roots with positive
imaginary part, which correspond to $\phi_n$ with positive real part:
$$
\Xi(v)=c_2\prod\left(1-\frac{v^2}{\phi_n^2}\right)
$$

We notice that $\Xi(v)$ is an even function, which is actually
a direct consequence of the functional equation
$\xi(s)=\xi(1-s)$.

What we would like particularly to remark is the following: we can start in the other way around writing the Hadamard expansion for $\xi(\frac{1}{2}+iv)$, then looking at the roots $\phi_n$
and $-\phi_n$ we see that by pairing the corresponding terms the exponentials cancel with each other. As a last step we can obtain an expression in $s$ by substituting back $v=i(\frac{1}{2}-s)$; we'll see it again later on.

Let us deal now with the generalizations of the zeta function, in particular the case of the Dirichlet L-functions.

We have that $\Lambda(\chi,s)\doteq(q/\pi)^{(s+e)/2}\Gamma((s+e)/2)L(\chi,s)$ is entire and has a
Hadamard product expansion (see e.g. \cite{da00} or \cite{ka92}). $\Lambda(\chi,s)$ satisfies a functional equation for Dirichlet $L$-functions
(see e.g \cite[Theorem 10.2.14]{co07}
)
: 
$$
\Lambda(\chi,s)=W(\chi)\Lambda(\bar{\chi},1-s)
$$
where $W(\chi)$ is a complex number with
absolute value equal to $1$.



When we deal with product expansions of
$\Lambda(\chi,s)$, we know that we cannot easily drop the exponentials inside
the product (see also \cite{la06b}: one would need
some restrictions on the locations of the zeros), because, unless $\chi$ is a real character, the roots of $L(\chi,s)$
are only symmetric with respect to the critical line. Though, when one is considering the distribution of zeros, it might be useful to consider the following.



First notice that for the set of real characters the zeros are symmetric with
respect to the real axis and $W(\chi)=1$.

For the general case consider the function
$G(\chi,s)\doteq\Lambda(\chi,s)\Lambda(\bar{\chi},s)$ and
$\Psi(\chi,v)\doteq G(\chi,1/2+iv)$. If $s$ is not a
zero for $G$, then it is neither for $\Lambda(\chi,s)$ nor for
$\Lambda(\bar{\chi},s)$, and if $G$ has a finite number of roots
in a region, then this holds for the L-functions too; and zeros
of $G$ are symmetric also with
respect to the real axis. In fact, because of the functional equation
for $L(\chi,s)$, and the fact that
$\overline{W(\chi)}=W(\bar{\chi})$, we have
$G(\chi,s)=G(\chi,1-s)$. Moreover $G(\chi,s)$ is real for $s$
real (or $\Psi(\chi,v)$ is real for imaginary $v$):
$$
\overline{\Lambda(\chi,s)\Lambda(\bar{\chi},s)}=\Lambda(\bar{\chi},s)\Lambda(\chi,s)=\Lambda(\chi,s)\Lambda(\bar{\chi},s).
$$
We are now in the position to obtain a modified Hadamard
product like that of $\zeta$, as explained in \cite[Theorem 4]{la06b}. One can also see that $B=0$ by the evenness of $\Psi$,
which again follows from the functional equation for $G$.

But in fact we can do something nicer: we don't need groups of four roots to obtain a modified product, provided that we consider doing a change of variable as above. 
We would consider a function
$$
\Xi(\chi,v)=e^{A+Bv}\prod_{n=1}^{\infty}\left(1-\frac{v}{\phi_n}\right)e^{v/\phi_n},
$$
where again zeros $\rho_n$ of $\Lambda(\chi,s)$ are written as
$\frac{1}{2}+i\phi_n$, with $\phi_n$ zeros of $\Xi(\chi,v)$.

Since the product with the
exponentials is already converging absolutely and the order of the
factors does not matter, we can pair, as above, $\phi_n$ and
$-\phi_n$ and the exponentials drop.

Same scenario holds for automorphic $L$-functions (see \cite[section 2, in particular Theorem 2.1]{la07}) and $L$-functions associated to
Hecke characters (see e.g. \cite{la94} and \cite{fr06}), and in
particular the Dedekind zeta function which corresponds to the
trivial character. But anyway the Dedekind zeta function is real on the real axis, as
the norm of ideals are positive integers: then,
because of the Schwarz reflection principle for analytic functions,
its roots are symmetric with respect to the real axis, too.


\section{Sufficient and equivalent criteria}

In this section we show that many of the statements in \cite{sc10a} can be generalized to other scenarios, in particular that of the Dirichlet L-functions; as said, we have no ambition to be exhaustive.

We first consider analogues of Theorem 2.1 and Corollary 2.4 in \cite{sc10a}: the proof given makes use of some important ingredients which we now have a deeper look at.

First, the argument proceeds by contradiction, using the fact that a function having zeros in some region cannot be approximated uniformly by $\zeta$ in Voronin's sense. But in fact $\zeta$ shares this property, the impossibility of that kind of
approximation, as it is stated in \cite{st07}, with universal $L$-functions in general. The argument given in \cite{st07} though seems to need essentially a density
estimate of the type $N(1/2+\epsilon, T)=o(T)$ for any
$\epsilon>0$, where $N(\sigma, T)$ stands for
the number of zeros in the upper critical strip up to height $T$ and with real
part bigger or equal than $\sigma$. This is known
to hold in many cases, but still not in general, as far as we know
(\cite{mo71b}, \cite{se92a},
\cite{ka95}, \cite{hi76}, \cite{st07}, \cite{he77a}, \cite{re80}, \cite{pe82a}).
In particular from $N(\sigma, T)= O(T^{1-\alpha(\sigma-1/2)\log
  T})$, which holds for zeta and
Dirichlet $L$-functions as well as for $L$-functions of quadratic
fields and Dirichlet series connected to some cusp forms
(\cite{se92a}), one can deduce the density estimate above.

A second ingredient we used was the universality of the derivative of $\zeta$, and again we have this property as well in the case of Dirichlet $L$-functions (more generally we refer to \cite{ba82a}, \cite[section 1.3 and 1.6]{st03}, \cite{la85c}, \cite{la05}).

And the third important property we used was a result by Selberg: for Dirichlet $L$-functions it is Selberg himself who provides an analogue (\cite[section 4]{se92a}): if $[0,T]$ is
divided into intervals of
length $\eta$, all except $o(T/\eta)$ of them will contain
$c\eta\log T+O((\eta\log T)^{\frac{3}{4}})$ zeros, where $c$ is a
constant (depending more generally on the degree of the
$L$-function). And this implies that, for fixed $\eta$,
$$
\lim_{T\to\infty} \frac{1}{T}\mu\left\{t\in(0,T] : \exists k\in(t,t+\eta) \mid
L(\chi,1/2+ik)=0 \right\}=1.
$$

We take the opportunity to remark also the similarity with a result by Fujii (\cite{fu75}) in the case of $\zeta$, which gives also some bounds about the number of zeros in
short intervals: if $\eta\log T$ tends to $\infty$ with $T$ and $N(T)$ is the number
of roots of $\zeta$ in the upper critical strip up to height $T$, then
$$
\frac{(\eta/2\pi)\log T}{\Phi(T)}< N(t+\eta)-N(t)<\Phi(t)(\eta/2\pi)\log T
$$
for any positive increasing function $\Phi(T)$ going to $\infty$
with $T$ and for almost all (in the sense of density) $t$ with $0<t\leq T$.

So, at least for the case of Dirichlet $L$-functions, we have what is needed to adapt the proof in \cite{sc10a} and obtain:

\begin{thm}\label{suff2}
Suppose that $L'(\chi,s)$ is strongly universal for any compact subset of the region $\{\frac{1}{2}<\sigma<1\}\cup\{[s_*,s_*+i\eta]\}$, where $s_*$ is a zero on the critical line and $\eta$ is a positive real number. Then the Extended Riemann Hypothesis is true.
\end{thm} 

In \cite{sc10a} we remarked that a direct argument could have been used instead of proceeding by contradiction: for generalizations of this argument, that is for connections of universality theorems with correspondent
Riemann hypotheses we refer to \cite[section 8.3]{st07}. In this setting it is an issue to know which classes of functions admit analogues of
universality theorems: a good resource for this is \cite{st07} and in
particular Theorem 6.12 combined with the table of degrees in
\cite{fr06}.

\vspace{1cm}

The next result we generalize concerns the equivalence with non-vanishing of partial sums, as stated in \cite[section 3]{sc10a}. 

\begin{thm}
The Extended Riemann Hypothesis for nonprincipal characters is true if and only if for any compact disc $K$ in the right (or left)
open half of the critical strip there exists an $N_0$ such
that, for infinitely many $N>N_0$, $\sum_{n=1}^N\frac{\chi(s)}{n^s}$ is non-vanishing for all $s$ in
$K$.
\end{thm}
\begin{proof}
We already said that the Dirichlet series $\sum_{n=1}^{\infty}\frac{\chi(s)}{n^s}$ is convergent for $\Re(s)>0$, except in the case of a principal character. So excluding this case, which means excluding the case of the classical Riemann zeta function, the proof of \cite[Theorem 3.2]{sc10a} can be adapted to this scenario.
\end{proof}

For an analogue of \cite[Theorem 3.1]{sc10a}, together with a corresponding expression for $H_N$ and an analogue of \cite[item (8)]{sc10a}, we notice the following.

We remind the reader (see e.g. \cite{sc10b}) that if $\chi$ is a character mod $q$, then
$$
L(\chi,s)=q^{-s}\sum_{a=1}^q \chi(a)\zeta(s,a/q),
$$
where $\zeta(s,a/q)$ are Hurwitz zeta functions, and that
$$
\zeta(s,w)=\sum_{n=0}^{N-1}(n+w)^{-s}+\frac{(N+w)^{1-s}}{s-1}+\frac{1}{2}(N+w)^{-s}+O((N+w)^{-\sigma-1})
$$
uniformly on compact subsets of $\mathbb{C}\backslash\{1\}$.

Then we can define $G_N(s)\doteq q^{-s}\sum_{a=1}^q \chi(a)\left(\sum_{n=0}^{N-1}(n+a/q)^{-s}+\frac{(N+a/q)^{1-s}}{s-1}\right)$ and state
\begin{thm}
The Extended Riemann Hypothesis is true if and only if for any compact disc $K$ in the right (or left)
open half of the critical strip there exists an $N_0$ such
that, for infinitely many $N>N_0$, $G_N(s)$ is non-vanishing for all $s$ in
$K$.
\end{thm}
\begin{proof}
We can adapt the proof in \cite[Theorem 3.1]{sc10a}, as the $G_N(s)$ are also converging uniformly on compact subsets of $\mathbb{C}\backslash\{1\}$.
\end{proof}

\vspace{1cm}

Finally we discuss about generalizations of Theorem 3.4 and Proposition 3.6 in \cite{sc10a} to the case of Dirichlet $L$-functions.
As we are not aware of general theoretical or computational results on zero-free regions within a certain height in the critical strip, we cannot really generalize Theorem 3.4 as it is, without adding additional conditions.
Notice however the following analogies. First $\overline{L(\chi,s)}=L(\bar{\chi},\bar{s})$, so that, when $\sigma=\frac{1}{2}$, $\left|\frac{L(\bar{\chi},s)}{L(\chi,1-s)}\right|$ equals $1$. When $\sigma\neq\frac{1}{2}$ and $t\geq 2\pi+1$, we refer to \cite[Theorem 5]{sa03}, which generalizes \cite[Theorem 1]{sa03} that we already mentioned in the classical case. And for the region below we have again an analogue to \cite[Proposition 3.6]{sc10a}:



\begin{thm}\label{teo2}
$\left|\frac{L(\bar{\chi},s)}{L(\chi,1-s)}\right|\neq 1$ in the regions $\{0<\sigma<\frac{1}{2}\}\cap\{\sqrt{(1+\sigma)^2+t^2}< 2\pi\}$ and $\{\frac{1}{2}<\sigma<1\}\cap\{\sqrt{(2-\sigma)^2+t^2}<
2\pi\}$.
\end{thm}
\begin{proof}
We only point out the essential modification with respect to the proof given for the classical case: if $\chi(-1)=-1$, we need a different inequality, namely $\left|\frac{\Gamma(\frac{2-s}{2})}{\Gamma(\frac{s+1}{2})}\right|\leq|(1+s)/2|^{1/2-\sigma}$ (\cite[Lemma 2]{ra59} specialized with $q=1$).

\end{proof}

\bibliography{huge} \bibliographystyle{plain}
\end{document}